\documentclass[12pt]{article}
\usepackage{a4,amsmath,amsthm,amssymb,enumerate,eucal,graphicx,
  subfig,xcolor,array}
\title{\textbf{The hamiltonicity\\of essentially 9-connected line
    graphs}}%
\author{Tom\'{a}\v{s} Kaiser$^{\:1}$\\
  Petr Vr\'{a}na$^{\:1}$}
\date{}

\newtheorem{theorem}{Theorem}
\newtheorem{lemma}[theorem]{Lemma}
\newtheorem{proposition}[theorem]{Proposition}
\newtheorem{claim}{Claim}

\newtheorem{observation}[theorem]{Observation}

\newcommand\size[1] {\left|{#1}\right|}

\newcommand\Setx[1] {\left\{{#1}\right\}}

\newcommand{\PP}{\mathcal P}%
\newcommand{\RR}{\mathcal R}%
\newcommand{\SSS}{\mathcal S}%
\newcommand{\bd}[2]{\partial_{#1}(#2)}%
\newcommand{\mnbr}[1]{N_*(#1)}%
\newcommand{\Ge}{G^e}%
\newcommand{\He}{H^e}

\newcommand{\xcore}[1]{#1^\star}
\newcommand{\xhyper}[1]{#1^+}
\newcommand{\eps}{\varepsilon}

\newcommand{\fig}[1]{\includegraphics[page=#1]{ess9-figs.pdf}}%
\newcommand{\sfig}[2]{\subfloat[#2]{\fig{#1}}}%

\newcommand{\sfigtop}[2]{\newbox{\pic}\sbox{\pic}{\fig{#1}}%
  \subfloat[#2]{\vbox to\ht\base{\hbox to\wd\pic{\usebox\pic}}}}
\newcommand{\hf}{\hspace*{0pt}\hspace*{\fill}\hspace*{0pt}}

\newcommand{\claimproofend}{\hspace*{.1mm}\hspace{\fill}}
\newenvironment{claimproof}{}{\claimproofend\par\vspace{2mm}}
\newtheorem{xcasehdr}{Case}
\newtheorem{xsubcasehdr}{Subcase}[xcasehdr]

{\vspace{-1mm}\par\noindent\xcasehdr{#1}\upshape\vspace{2mm}\par\noindent}%
{\hspace*{0mm}\hspace{\fill}$\diamond$\par\vspace{2mm}}

{\vspace{-1mm}\par\noindent \xsubcasehdr{#1}\upshape
  \vspace{2mm}\par\noindent}%
{\hspace*{0mm}\par\vspace{2mm}}

\begin{document}
\maketitle
\footnotetext[1]{Department of Mathematics and European Centre of
  Excellence NTIS (New Technologies for the Information Society),
  University of West Bohemia, Univerzitn\'{\i}~8, 306~14~Plze\v{n},
  Czech Republic. Supported by project 17-04611S of the Czech Science
  Foundation. E-mail:
  \{\texttt{kaisert,vranap}\}\texttt{@kma.zcu.cz}.}

\begin{abstract}
  Yang et al. proved that every 3-connected, essentially 11-connected
  line graph is Hamilton-connected. This was extended by Li and Yang
  to 3-connected, essentially 10-connected graphs. Strengthening their
  result further, we prove that 3-connected, essentially 9-connected
  line graphs are Hamilton-connected. We use a method based on
  quasigraphs in combination with the discharging technique. The
  result extends to claw-free graphs.
\end{abstract}


\section{Introduction}
\label{sec:introduction}

A conjecture of Thomassen~\cite{Tho-reflections} states that every
4-connected line graph is hamiltonian. Motivated by this conjecture,
Lai et al.~\cite{LSWZ-every} studied the hamiltonicity of line graphs
in relation to their essential connectivity (the definition is
recalled later in this section). They proved the following result:

\begin{theorem}\label{t:lai}
  Every 3-connected, essentially 11-connected line graph is
  hamiltonian.
\end{theorem}

In fact, Yang et al.~\cite{YLLG-collapsible} prove that under the
assumption of Theorem~\ref{t:lai}, the graph is
\emph{Hamilton-connected}, i.e., any pair of its vertices is joined by
a Hamilton path. The result was recently improved by Li and
Yang~\cite{LY-every} who showed that 3-connected, essentially
10-connected graphs are Hamilton-connected.

This partially answered a question of Lai et al.~\cite{LSWZ-every}
whether the constant in Theorem~\ref{t:lai} can be replaced by a
smaller one. They note that the least possible value is 4, which is
improved to 5 in \cite{YXLG-hamiltonicity}. We add that there are
3-connected, essentially 4-connected line graphs that do not even have
a Hamilton path, such as the line graph of a graph obtained by adding
one pendant edge to each vertex of the graph in
Figure~\ref{fig:example}a.

Regarding the assumption of 3-connectedness in Theorem~\ref{t:lai}, we
remark that there are 2-connected line graphs of arbitrary essential
connectivity that do not have any Hamilton path. A class of examples
can be constructed by replacing each edge of the complete graph on 4
vertices by an odd number of internally disjoint paths of length 3, as
shown in Figure~\ref{fig:example}b.

\begin{figure}
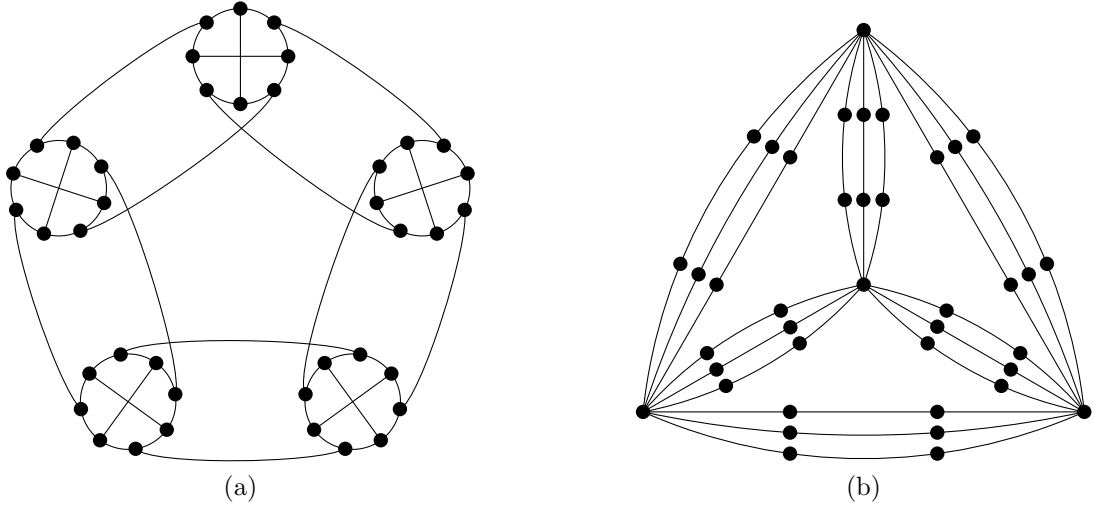

  \centering
  \sfig{21}{}\hf\sfig{18}{}
  \caption{(a) A graph used to construct a 3-connected, essentially
    4-edge-connected line graph without a Hamilton path. (b) A graph
    $G$ such that $L(G)$ is 2-connected and essentially $k$-connected
    (for $k=9$) and has no Hamilton path.}
\label{fig:example}
\end{figure}

The main result of the present paper is a further strengthening of
Theorem~\ref{t:lai} as follows:

\begin{theorem}\label{t:main}
  Every 3-connected, essentially 9-connected line graph is
  Hamilton-connected.
\end{theorem}

Our approach uses a strengthening of the main result
of~\cite{KV-hamilton}, proved in a companion paper~\cite{KV-quasi}. In
particular, we use a reduction to hypergraphs, which is described in
Section~\ref{sec:hyper}. Preliminaries on the Hamilton connectivity of
line graphs (and the implications on the hypergraph side of the
problem) are given in Section~\ref{sec:hamilt-conn}.

Section~\ref{sec:quasi} describes quasigraphs, another crucial
component of the approach of~\cite{KV-hamilton,KV-quasi}, and states
the main technical tool of this paper, Theorem~\ref{t:skeletal} proved
in~\cite{KV-quasi}. A counting argument based on the outcome of
Theorem~\ref{t:skeletal} is given in Section~\ref{sec:counting}.

Section~\ref{sec:structural} is essentially a study of small
configurations in the hypergraph corresponding to a graph satisfying
the hypothesis of Theorem~\ref{t:main}. It prepares the way for the
application of a discharging argument in
Section~\ref{sec:discharging}. The final section is devoted to an
extension of the result to claw-free graphs and concluding remarks.

In the remainder of this section, we recall the necessary terminology
and notation. For graph-theoretical concepts not explained here, see,
for example, Diestel~\cite{Die-graph}. Unless otherwise noted, all
graphs in this paper are loopless and may contain parallel edges. When
speaking of a line graph of a graph $G$, it is understood that $G$ may
contain parallel edges (i.e., $G$ may be a multigraph), but $L(G)$ is
by definition a simple graph.

Let $G$ be a graph. We write $V(G)$ and $E(G)$ for the set of vertices
and the set of edges of $G$, respectively. For $i\geq 0$, $V_i(G)$ is
the set of vertices of degree $i$ in $G$.

A vertex-cut or edge-cut $X$ is called \emph{essential} if $G-X$ has
at least two components which are \emph{nontrivial} (i.e., contain
more than one vertex). The graph $G$ is \emph{essentially
  $k$-connected} if it has more than $k$ vertices and contains no
essential vertex-cut of size less than $k$. Similarly, $G$ is
\emph{essentially $k$-edge-connected} if it contains no essential
edge-cut of size less than $k$. It is not hard to see that $L(G)$ is
$k$-connected if and only if $G$ is essentially $k$-edge-connected and
$\size{E(G)} > k$.

We extend the above definitions as follows. An edge-cut $X$ in a graph
$G$ is \emph{$r$-essential} ($r\geq 0$) if there at least two
components of $G-C$, each of which contains at least $r$ edges. The
graph $G$ is \emph{$r$-essentially $k$-edge-connected} ($k\geq 1$) if
it has no $r$-essential edge-cuts of size less than $k$. Thus,
`$0$-essentially $k$-edge-connected' is the same as
`$k$-edge-connected', while `$1$-essentially $k$-edge-connected' is
the same as `essentially $k$-edge-connected'.

We note the following easy observation:
\begin{observation}\label{obs:2-ess}
  The line graph $L(G)$ of a graph $G$ is essentially $k$-connected if
  and only if $G$ is $2$-essentially $k$-edge-connected and
  $\size{E(G)} > k$.
\end{observation}

The \emph{length} of a path is the number of its edges. The degree of
a vertex $v$ in a graph $G$ is denoted by $d_G(v)$. Given a set of
vertices $X\subseteq V(G)$, we let $\bd G X$ denote the set of edges
of $G$ with exactly one endvertex in $X$. Furthermore, we extend the
notation for the degree of a vertex and set $d_G(X) = \size{\bd G X}$.

Besides graphs, we will also consider \emph{3-hypergraphs}, that is,
hypergraphs with all hyperedges of size 2 or 3 (called
\emph{2-hyperedges} or \emph{3-hyperedges} accordingly). For a
hypergraph $H$, $V(H)$ and $E(H)$ denote its vertex set and its
hyperedge set, respectively. The symbol $V_i(H)$ denotes the set of
vertices of degree $i$ ($i\geq 0$). In addition, for $i\in\Setx{2,3}$,
we let $E_i(H)$ denote the set of $i$-hyperedges of $H$. We define a
graph $G(H)$ whose vertex set is $V(H) \cup E_3(H)$, with vertices
$u,v$ joined by an edge if either $u,v$ are neighbours in $H$, or
$v\in E_3(H)$ and $u$ is a vertex contained in $v$.

For $X\subseteq V(H)$, we let $\bd H X$ be the set of hyperedges of
$H$ intersecting $X$ but not contained in it, and define $d_H(X) =
\size{\bd H X}$ as in the graph case.

If $e$ is a hyperedge of $H$ and $v$ is a vertex contained in $e$,
then the \emph{detachment} of $e$ from $v$ is the operation which
removes $e$ from $H$ and, in case $\size e=3$, replaces it with
$e-\Setx v$.


\section{Reduction to hypergraphs}
\label{sec:hyper}

From this point on, let $G$ be a graph whose line graph satisfies the
hypothesis of Theorem~\ref{t:main}. By Observation~\ref{obs:2-ess} and
the preceding discussion, $G$ is essentially 3-edge-connected and
2-essentially 9-edge-connected.

We begin by transforming $G$ to a graph $G^0$ called the \emph{core}
of $G$. Recall that the \emph{suppression} of a vertex $v$ of degree 2
is the contraction of one of its incident edges (discarding any loop
that may result). The graph $G^0$ is obtained from $G-V_1(G)$ by
suppressing all the vertices of degree 2. By our assumption that $G$
is essentially 3-edge-connected, $G-V_1(G)$ has minimum degree at
least 2 unless $G$ is a star. (Note also that that $V_2(G-V_1(G)) =
V_2(G)$.) Clearly, either $G^0$ is \emph{trivial} (i.e., it has only
one vertex), or $G^0$ has minimum degree at least 3. This observation
is strengthened in Lemma~\ref{l:g0} below.

Given a set $X\subset V(G^0)$, we define $\xcore X$ as the union of
$X$ with the set of vertices $y\in V(G)-V(G^0)$ such that $N_G(y)
\subseteq X$.

\begin{observation}
  \label{obs:cuts}
  For any $X\subseteq V(G^0)$, it holds that $d_{G^0}(X) = d_G(\xcore
  X)$. Furthermore, if $\bd{G^0} X$ is an $r$-essential edge-cut in
  $G^0$ ($r \geq 0$), then $\bd G {\xcore X}$ is an $r$-essential
  edge-cut in $G$.
\end{observation}

\begin{lemma}
  \label{l:g0}
  If $L(G)$ is 3-connected and essentially 9-connected, then $G^0$ has
  the following properties:
  \begin{enumerate}[\quad(i)]
  \item $G^0$ is 3-edge-connected,
  \item $G^0$ is essentially 4-edge-connected,
  \item $G^0$ is 2-essentially 9-edge-connected.
  \end{enumerate}
\end{lemma}
\begin{proof}
  Part (i) was proved by Shao~\cite{Sha-claw}
  (see~\cite[Lemma~2.2(i)]{LSWZ-every}). Part (iii) of the lemma
  follows from Observation~\ref{obs:cuts}. We prove part (ii). For
  contradiction, let $F$ be an essential edge-cut in $G^0$ of size at
  most 3. Let $K$ and $L$ be two components of $G^0-F$ containing at
  least one edge each. Since $F$ is not 2-essential, one of them (say,
  $K$) contains exactly one edge. The assumption that $\size F \leq 3$
  implies that one of the two vertices of $K$ has degree at most 2, a
  contradiction with part (i).
\end{proof}

The vertices of $G$ which are not vertices of $G^0$ are called
\emph{transient}. A vertex $v$ of $G^0$ with $d_{G^0}(v) = 3$ is said
to be \emph{protected} if it is adjacent in $G$ to a transient vertex.

We now turn $G^0$ into a 3-hypergraph $H^0$. Let $W_3$ be the set of
vertices of degree 3 in $G^0$, and let $W_3^\times$ be the subset of
$W_3$ consisting of the protected vertices. We choose a maximal
independent subset $W$ of $W_3-W_3^\times$. The 3-hypergraph $H^0$ is
constructed from $G^0-W$ by adding, for each $w\in W$, a hyperedge
$h(w)$ consisting of the neighbours of $w$ in $G^0$ provided that
there are at least two such neighbours; if there are exactly two, then
$\size{h(w)} = 2$. The procedure is illustrated in
Figure~\ref{fig:hyper}. In all the figures in this paper, a
3-hyperedge is represented by three lines meeting at a point which is
not marked as a vertex.

\begin{figure}
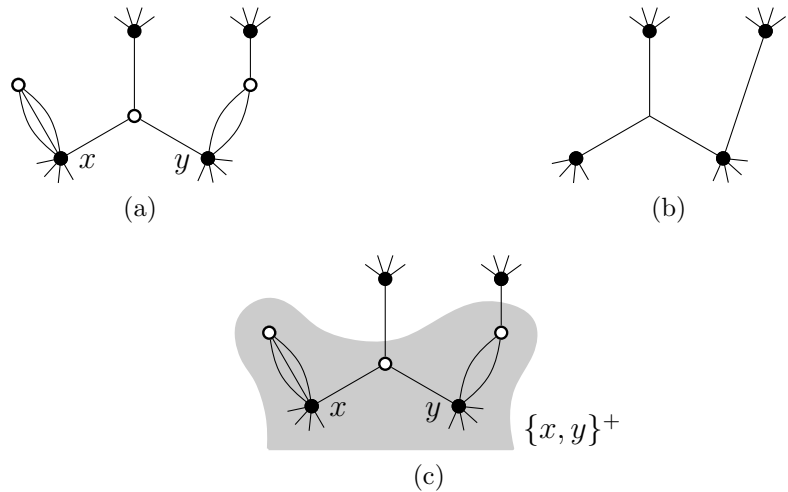

  \centering%
  \hf\sfig7{}\qquad\hf\qquad\sfig8{}\hf\\
  \qquad\sfig{16}{}
  \caption{The transformation of $G^0$ to $H^0$. (a) Part of the graph
    $G^0$ with vertices in $W$ filled white. (b) The corresponding
    part of $H^0$, where the three incident lines without a common
    vertex mark represent a 3-hyperedge. (c) The set $\Setx{x,y}^+$
    includes all of the vertices of $W$ shown in the picture.}
  \label{fig:hyper}
\end{figure}

The vertices in $W$ are called \emph{temporary}, the other vertices of
$G^0$ \emph{permanent}. Thus, the set of permanent vertices is the
vertex set of $H^0$. Note that all protected vertices are permanent.

\begin{lemma}\label{l:permanent}
  Every edge of $G$ has a permanent endvertex.
\end{lemma}
\begin{proof}
  Assume that an edge $e$ of $G$ has endvertices $u$ and $v$, and that
  $u$ is not permanent. If $u$ is transient, then $v$ is protected and
  hence permanent. Otherwise, $u$ is temporary, in which case $v$ is
  neither transient (as this would make $u$ permanent) nor temporary
  (since temporary vertices form an independent set). Hence, $v$ is
  permanent.
\end{proof}

Similarly to the set $\xcore X$ above, we introduce a set $\xhyper Y
\subseteq V(G^0)$, where $Y\subseteq V(H^0)$. The definition is
illustrated in Figure~\ref{fig:hyper}c. The set $\xhyper Y$ is
defined as the union of $Y$ with the set of all the vertices $w\in W$
such that in $G^0$, $w$ is incident with at least two edges to $Y$
(possibly parallel). As in Observation~\ref{obs:cuts}, we have
$d_{H^0}(Y) = d_{G^0}(\xhyper Y)$. Since $G^0$ is 3-edge-connected, we
obtain the following observation:

\begin{observation}\label{obs:3ec}
  The hypergraph $H^0$ is 3-edge-connected.
\end{observation}

A more detailed study of the properties of $H^0$ is undertaken in
Section~\ref{sec:structural}. The results will be used in the design
of a discharging procedure in Section~\ref{sec:discharging}.

We add one more definition. To each edge $e$ of $G$, we want to assign
a hyperedge $k(e)$ of $H^0$ which `corresponds' to $e$ in
$H^0$. First, let us define a value $k_1(e)$ which is either an edge
of $G^0$, or the empty set:
\begin{equation*}
  k_1(e) =
  \begin{cases}
    \emptyset & \text{if $e$ is incident with $V_1(G)$,}\\
    e' & \text{if $e$ is incident with a transient vertex $z\notin V_1(G)$}\\
    & \text{\quad and $e'$ is obtained by suppressing $z$,}\\
    e & \text{otherwise.}
  \end{cases}
\end{equation*}
Observe that $k_1(e)$ is well-defined, because the assumption that $G$
is essentially 3-edge-connected implies that $V_1(G) \cup
V_2(G-V_1(G))$ is an independent set.

Next, we associate a hyperedge $k_2(f)$ of $H^0$ with each edge $f$ of
$G^0$. In the definition, we allow $f$ or $k_2(f)$ to be the empty
set.
\begin{equation*}
  k_2(f) =
  \begin{cases}
    \emptyset & \text{if $f = \emptyset$}, \\
    h(w) & \text{if $f$ is incident with a temporary vertex $w$,} \\
    f & \text{otherwise.}
  \end{cases}
\end{equation*}
Finally, for an edge $e$ of $G$, we define
\begin{equation*}
  k(e) = k_2(k_1(e)).
\end{equation*}


\section{Hamilton connectivity of line graphs}
\label{sec:hamilt-conn}

For the cases where the core $G^0$ is small, we will be able to verify
Theorem~\ref{t:main} directly using the following result
\cite[Lemma~3.3]{LLS-s-hamiltonian}:
\begin{lemma}\label{l:trees}
  If $L(G)$ is 3-connected and $G^0$ contains two edge-disjoint
  spanning trees, then $L(G)$ is Hamilton-connected.
\end{lemma}
Lemma~\ref{l:trees} will be used in conjunction with the
characterization of graphs with two disjoint spanning trees, which
follows from a more general result of Tutte~\cite{Tut-problem} and
Nash-Williams~\cite{NW-edge}:
\begin{theorem}[Tutte and Nash-Williams]\label{t:tnw}
  The graph $G^0$ has two edge-disjoint spanning trees if and only if
  for every partition $\PP$ of $V(G^0)$, the number of edges of
  $V(G^0)$ with endvertices in different classes of $\PP$ is at least
  $2(\size\PP - 1)$.
\end{theorem}

Using Lemma~\ref{l:trees} and Theorem~\ref{t:tnw}, we obtain the
following:
\begin{lemma}\label{l:small}
  If $G^0$ has at most 5 vertices, then $L(G)$ is Hamilton-connected.
\end{lemma}
\begin{proof}
  We use Theorem~\ref{t:tnw} to show that $G^0$ admits two
  edge-disjoint spanning trees. Let $\PP$ be a partition of
  $V(G^0)$ and let $F$ be the set of edges with endvertices in
  different classes of $\PP$. By Lemma~\ref{l:g0}(i), $G^0$ is
  3-edge-connected and thus $\size F \geq 3\size\PP/2$. Since
  $\size\PP\leq 5$, we have $\lceil 3\size\PP/2\rceil \geq
  2(\size\PP-1)$. The lemma follows from Theorem~\ref{t:tnw} and
  Lemma~\ref{l:trees}.
\end{proof}
For the purposes of this paper, Lemma~\ref{l:small} enables us to
restrict ourselves to the case that $G^0$ has at least 6 vertices.

We recall a well-known necessary and sufficient condition for the line
graph $L(G)$ to be Hamilton-connected. Let $e_1,e_2$ be edges of
$G$. A trail $T$ in $G$ is an \emph{$(e_1,e_2)$-trail} if its first
edge is $e_1$ and its last edge is $e_2$. The trail $T$ is
\emph{internally dominating} if every edge of $G$ is incident with an
internal vertex of $T$. Similarly, $T$ is \emph{internally spanning}
if every vertex of $G$ appears as an internal vertex of $T$. The
following is a folklore analogue of Harary and Nash-Williams'
characterization of hamiltonian line graphs
(cf.~\cite[Theorem~1.5]{LLZ-eulerian}).

\begin{theorem}\label{t:ham-conn-preimage}
  Let $G$ be a graph with at least 3 edges. Then $L(G)$ is
  Hamilton-connected if and only if for every pair of edges $e_1,e_2\in
  E(G)$, $G$ has an internally dominating $(e_1,e_2)$-trail.
\end{theorem}

We will infer the existence of internally dominating trails in $G$
using hypergraphs obtained by a small modification of $H^0$ (the
3-hypergraph associated with $G$ as in
Section~\ref{sec:hyper}). First, we define a 3-hypergraph
$\He$ and a pair of vertices $a_1,a_2$ of $\He$. The
edges $e_i$ and the vertices $a_i$ ($i=1,2$) will be considered fixed
throughout the paper.

For $i=1,2$, let $a_i$ be a permanent vertex of $e_i$ (which exists by
Lemma~\ref{l:permanent}). If possible, we choose $a_1$ and $a_2$ so as
to be distinct. The hypergraph $\He$ is obtained from $H^0$ by
detaching $k(e_1)$ from $a_1$ and, subsequently, detaching $k(e_2)$
from $a_2$. We also define $\Ge$ as the graph $G(\He)$ corresponding
to $\He$. Note that every permanent vertex of $G$ is a vertex of
$\Ge$.

We say that a trail $T$ in a graph $G'$ \emph{spans} a set $X\subseteq
V(G')$ if $X\subseteq V(T)$. If the first and last vertices of $T$ are
$a$ and $b$, respectively, we say that $T$ is an
\emph{$ab$-trail}. The case $a=b$ is allowed in this definition.

The following lemma provides a bridge between spanning $a_1a_2$-trails
in $G(\He)$ and internally dominating $(e_1,e_2)$-trails in $G$:

\begin{lemma}\label{l:join}
  If $\Ge$ admits an $a_1a_2$-trail spanning $V(\He)$, then $G$
  contains an internally dominating $(e_1,e_2)$-trail.
\end{lemma}
\begin{proof}
  Let $T^e$ be an $a_1a_2$-trail in $\Ge$ spanning $V(\He)$. Let $T$
  be the corresponding $a_1a_2$-trail in $G$ (that is, whenever $T^e$
  uses an edge $f$ obtained by suppressing a vertex $w$, $T$ uses the
  two edges incident with $w$). Since $a_i$ is an endvertex of $e_i$
  ($i=1,2$), we can construct an $(e_1,e_2)$-trail $T'$ in $G$ by
  prepending $e_1$ and appending $e_2$ to $T$. Since $T$ spans all
  permanent vertices and every edge of $G$ has a permanent endvertex,
  $T'$ is internally dominating.
\end{proof}

In view of Lemma~\ref{l:join}, proving Theorem~\ref{t:main} reduces to
finding an $a_1a_2$-trail spanning $V(\He)$ in $\Ge$ for each choice
of $e_1,e_2$. A basic tool for this is Proposition~\ref{p:qt-join} in
Section~\ref{sec:quasi}.


\section{Quasigraphs}
\label{sec:quasi}

Our proof relies on a strengthening of the so-called Skeletal
Lemma~\cite[Lemma 17]{KV-hamilton}. The required stronger version is
proved in~\cite{KV-quasi}. The formulation and proof of this result
use the language and some theory of quasigraphs; in this section, we
recall just the bare minimum allowing us to state
Theorem~\ref{t:skeletal}.

Recall that a \emph{3-hypergraph} is a hypergraph whose edges have
size 2 or 3. Let $H$ be a 3-hypergraph.

A \emph{quasigraph} in $H$ is a mapping $\pi$ that assigns to each
hyperedge $e$ of $H$ either a subset of $e$ of size 2, or the empty
set. The hyperedges $e$ with $\pi(e)\neq \emptyset$ are said to be
\emph{used} by $\pi$.

Let $\pi$ be a quasigraph in $H$. We let $\pi^*$ denote the graph on
$V(H)$, obtained by considering the pairs $\pi(e)$ ($e\in E(H)$) as
edges whenever $\pi(e)\neq\emptyset$. If $\pi^*$ is a forest, then
$\pi$ is said to be \emph{acyclic}. If $\pi^*$ is the union of a circuit
and a set of isolated vertices, then $\pi$ is a \emph{quasicycle}. The
hypergraph $H$ is \emph{acyclic} if there exists no quasicycle in
$H$.

Let $X\subseteq V(H)$. We say that the quasigraph $\pi$ is
\emph{connected} on $X$ if the induced subgraph of $\pi^*$ on $X$ is
connected. A somewhat more involved notion is anticonnectedness: we
say that $\pi$ is \emph{anticonnected} on $X$ if for each nontrivial
partition $\RR$ of $X$, there is a hyperedge $f$ such that $f$
intersects at least two classes of $\RR$ and $\pi(f)$ is contained in
one of them.

Let $\PP$ be a partition of $V(H)$. If $e\in E(H)$, then $e/\PP$ is
defined as the set of all classes of $\PP$ intersected by $e$. If
there is more than one such class, then $e$ is said to be
\emph{$\PP$-crossing}. The hypergraph $H/\PP$ has vertex set $\PP$ and
its hyperedges are all the sets of the form $e/\PP$, where $e$ is a
$\PP$-crossing hyperedge of $H$. A quasigraph $\pi/\PP$ in this
hypergraph is defined by setting, for every $\PP$-crossing hyperedge
$e$ of $H$,
\begin{equation*}
  (\pi/\PP)(e/\PP) =
  \begin{cases}
    \pi(e)/\PP & \text{if $\pi(e)$ is $\PP$-crossing,}\\
    \emptyset & \text{otherwise.}
  \end{cases}
\end{equation*}

The \emph{complement} $\overline\pi$ of $\pi$ is the subhypergraph of $H$
(on the same vertex set) consisting of the hyperedges not used by
$\pi$.

A partition $\pi$ of $V(H)$ is \emph{$\pi$-skeletal} if both of the
following conditions hold:
\begin{enumerate}[(1)]
  \item for each $X\in\PP$, $\pi$ is both connected on $X$ and
    anticonnected on $X$,
  \item the complement of $\pi/\PP$ in $H/\PP$ is acyclic.
\end{enumerate}

For our purposes, the version of the Skeletal Lemma stated below in
Theorem~\ref{t:skeletal} needs to take care of a particular
configuration we call `bad leaf' since it presents a problem in our
computations. Let us describe this configuration.

Recall that $\pi$ is a quasigraph in a 3-hypergraph $H$. Assume now
that $\pi$ is acyclic. In each component of the graph $\pi^*$,
we choose an arbitrary root and orient all the edges of $\pi^*$ toward
the root. A hyperedge $e$ of $H$ is \emph{associated with} a vertex
$u$ if it is used by $\pi$ and $u$ is the tail of $\pi(e)$ in the
resulting oriented graph. Thus, every vertex has at most one
associated hyperedge, and conversely, each hyperedge is associated
with at most one vertex.

A vertex $u$ of $H$ is a \emph{bad leaf} for $\pi$ (and the given
choice of the roots of the components of $\pi^*$) if all of the
following hold:
\begin{enumerate}[\quad(i)]
\item $u$ is a leaf of $\pi^*$,
\item $u$ is incident with exactly three hyperedges, exactly one of
  which has size 3 (say, $e$), and
\item $e$ is associated with $u$.
\end{enumerate}

\begin{figure}
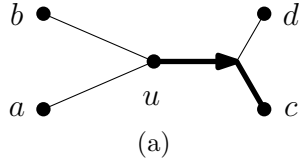

  \centering
  \hf
  \sfig1{}
  \hf
  \caption{A bad leaf $u$.}
  \label{fig:bad}
\end{figure}

Bad leaves can be eliminated at the cost of performing certain local
modifications in the hypergraph. More precisely, if $u$ is a vertex of
the $3$-hypergraph $H$ incident with exactly two hyperedges of size 2
and exactly one hyperedge of size 3, then a \emph{switch at $u$} is
the operation depicted in Figure~\ref{fig:switch}. (We remark that
in~\cite{KV-quasi}, the switch operation acts on quasigraphs in $H$ as
well, but this is not necessary for our purposes.)

We say that a $3$-hypergraph $\tilde H$ is \emph{related} to $H$ if it
can be obtained from $H$ by a finite sequence of switches at suitable
vertices. Note that, in this case, $G(\tilde H)$ is isomorphic to
$G(H)$.

\begin{figure}
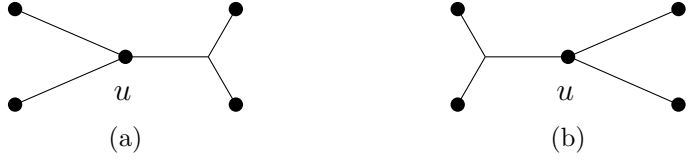

  \hf\subfloat[]{\fig{19}}\hf\subfloat[]{\fig{20}}\hf
  \caption{(a) A 3-hypergraph $H$ with a vertex $u$ suitable for a
    switch. (b) The hyperedges incident with $u$ in the 3-hypergraph
    resulting from the switch. The other hyperedges are not modified.}
  \label{fig:switch}
\end{figure}

We can now finally state the main technical result mentioned above, a
strengthening of the Skeletal Lemma proved in~\cite{KV-quasi}:

\begin{theorem}
  \label{t:skeletal}
  Let $H$ be a 3-hypergraph. There exists a hypergraph $\tilde H$
  related to $H$ and an acyclic quasigraph $\sigma$ in $\tilde H$ such
  that $\sigma$ has no bad leaves (for any choice of the roots of the
  components of $\sigma^*$) and $V(\tilde H)$ admits a
  $\sigma$-skeletal partition $\SSS$.
\end{theorem}

Theorem~\ref{t:skeletal} will be used in conjunction with the
following result, implied by a special case of Lemma 28
in~\cite{KV-hamilton}:
\begin{proposition}\label{p:qt-join}
  Let $H$ be a 3-hypergraph and $b_1,b_2$ vertices of $H$. If $H$
  admits an acyclic quasigraph that is both connected and
  anticonnected on $V(H)$, then the graph $G(H)$ contains a
  $b_1b_2$-trail spanning $V(H)$.
\end{proposition}
Proposition~\ref{p:qt-join} will be useful in Section~\ref{sec:final},
where we infer that the assumption of the proposition is satisfied,
which will enable us to apply Lemma~\ref{l:join}.


\section{Counting the hyperedges}
\label{sec:counting}

Recall that $G$ is a graph satisfying the assumptions of
Theorem~\ref{t:main}, and that $H^0$ is a 3-hypergraph associated with
the core $G^0$ of $G$. Additionally, $e_1,e_2$ are fixed edges of $G$,
$a_1,a_2$ are their permanent endvertices, and $\He$ is a
modification of $H^0$ defined in Section~\ref{sec:hamilt-conn}.

By Theorem~\ref{t:skeletal}, there is a 3-hypergraph $\tilde H$
related to $\He$ and an acyclic quasigraph $\sigma$ in $\tilde H$ with
no bad leaves such that $V(\tilde H)$ admits a $\sigma$-skeletal
partition $\SSS$. Our ultimate use of $\sigma$ is to find a connected
$X(e_1,e_2)$-join in the graph $G(\He)$ to be used in
Lemma~\ref{l:join}. Since $G(\tilde H)$ is isomorphic to $G(\He)$ if
$\tilde H$ and $\He$ are related, we may assume without loss of
generality that $\tilde H = \He$.

As we will see in Section~\ref{sec:final}, the proof of
Theorem~\ref{t:main} is simple in the case that $\size\SSS = 1$. In
the following calculations, we therefore assume $\SSS \geq 2$.

Recall that for a $3$-hypergraph $H$, an edge $e$ of $H$ and a
partition $\PP$ of $V(H)$, the notation $e/\PP$ and $H/\PP$ has been
defined in Section~\ref{sec:quasi}.

Let us write $d^0(P)$ for the degree of a vertex $P\in\SSS$ in the
hypergraph $H^0/\SSS$.

We proceed as in Section 8 of~\cite{KV-hamilton}. We set $\tau =
\sigma/\SSS$. Let $n = \size\SSS$ and let $m$ denote the number of
hyperedges of $H^e/\SSS$. For $k\in\Setx{2,3}$, let $m_k$ be the
number of $k$-hyperedges of $H^e/\SSS$ used by $\sigma$, and let
$\overline m_k$ denote the number of $k$-hyperedges of
$\overline\tau$. Since $\SSS$ is $\sigma$-solid, the graph $\tau^*$ is
acyclic. It has $n$ vertices and $m_2 + m_3$ edges, and hence
\begin{equation}
  \label{eq:1}
  m_2 + m_3 \leq n - 1.
\end{equation}
Similarly, the complement $\overline\tau$ is an acyclic
hypergraph. Consider the graph $G(\overline\tau)$, defined in
Section~\ref{sec:introduction}. Since $\overline\tau$ is acyclic, so
is $G(\overline\tau)$. As $G(\overline\tau)$ has $n+\overline m_3$
vertices and $\overline m_2 + 3\overline m_3$ edges, we get
\begin{equation}
  \label{eq:2}
  \overline m_2 + 2\overline m_3 \leq n - 1.
\end{equation}
Moreover, by the assumption, either $\tau^*$ or $G(\overline\tau)$ is
disconnected and therefore it has at most $n-2$ edges. Adding
\eqref{eq:1} to \eqref{eq:2} and using this fact, we find
\begin{equation}
  \label{eq:3}
  m \leq 2n-3-\overline m_3.
\end{equation}

For any hypergraph $H$, let $s(H)$ be the sum of vertex degrees in
$H$. Let us set
\begin{equation*}
  \eps = s(H^0/\SSS) - s(H^e/\SSS)
\end{equation*}
and observe that by the definition of the hypergraph $\He$,
$\eps \leq 4$. Furthermore, we have
\begin{align*}
  s(H^e/\SSS) &= 2(m_2 + \overline m_2) + 3(m_3 + \overline
  m_3) = 2m + (m_3 + \overline m_3) \\
  &\leq 4n - 6 + m_3 - \overline m_3,
\end{align*}
where the last inequality follows by using \eqref{eq:3} to substitute
for $m$. Substituting $s(H^0/\SSS) - \eps$ for $s(H^e/\SSS)$, we obtain
\begin{equation}
  \label{eq:s-h0}
  s(H^0/\SSS) \leq 4n-6 + (m_3 - \overline m_3) + \eps.
\end{equation}

The quasigraph $\tau$ in $H^e/\SSS$ determines an (acyclic) quasigraph
$\tau^0$ in $H^0/\SSS$ in a natural way. Let $m^0_3$ be the number of
3-hyperedges of $H^0/\SSS$ used by $\tau^0$ and observe that $m^0_3
\geq m_3$. Furthermore, let $\widetilde m^0_i$ ($i\in\Setx{2,3}$) be
the number of $i$-hyperedges of $H^0/\SSS$, whether used by $\tau^0$
or not.

Inequality~\eqref{eq:s-h0} has two corollaries. Firstly, using the
fact that $\eps \leq 4$ and $m^0_3 \geq m_3$, and ignoring the
$\overline m_3$ term, we find
\begin{equation}
  \label{eq:main}
  \sum_{P\in V(H^0/\SSS)}(d^0(P) - 4) - m^0_3 \leq -2.
\end{equation}

For the second corollary, note that since $s(H^0/\SSS) = 2\widetilde m^0_2
+ 3\widetilde m^0_3$ and $m_3 - \overline m_3 \leq \widetilde m^0_3$,
\eqref{eq:s-h0} implies
\begin{equation}
  \label{eq:small}
  \widetilde m^0_2 + \widetilde m^0_3 \leq 2n-3 + \frac\eps 2.
\end{equation}

For later use, we record an observation concerning classes $X$ of
$\SSS$ such that $\size{X^+} \geq 1$; let us call such classes
\emph{nontrivial}.
\begin{observation}
  \label{obs:nontriv}
  If $X$ is a nontrivial class of $\SSS$, then $G^0[X^+]$ is not a
  matching.
\end{observation}
\begin{proof}
  Suppose that $X\in\SSS$ is nontrivial. We prove that $G^0[X^+]$ has
  at least two incident edges. If $\size X > 1$, then this follows
  from the fact that $\sigma$ is both connected and anticonnected on
  $X$ by the choice of $\SSS$. On the other hand, if $\size X = 1$,
  then some vertex of $X^+-X$ is incident with at least two edges to
  $X$ and the assertion also holds.
\end{proof}


\section{Structural observations}
\label{sec:structural}

We continue to use the notation and assumptions of
Section~\ref{sec:counting}. The objective of this and the following
section is to rule out most cases in the proof of Theorem~\ref{t:main}
by establishing the following:
\begin{proposition}\label{p:S4}
  If $G^0$ has at least 6 vertices, then the partition $\SSS$ has at
  most 4 classes.
\end{proposition}

For the sake of a contradiction, let us assume that $\size\SSS \geq
5$. We now prove several claims concerning the structure of the
hypergraph $H^0/\SSS$ in this case.

\begin{lemma}\label{l:path}
  Suppose that $G^0$ has at least 6 vertices. Then the following hold:
  \begin{enumerate}[\quad(i)]
  \item $G^0$ contains no path of length 2 with two vertices of degree
    3 and one vertex of degree at most 4,
  \item no permanent vertex of degree 3 in $G^0$ is adjacent to a
    permanent vertex of degree at most 4.
  \end{enumerate}
\end{lemma}
\begin{proof}
  (i) We prove that if $x_1x_2x_3$ is a path in $G^0$, then
  $\sum_{i=1}^3 d_{G^0}(x_i) \geq 11$. Suppose the contrary. Define
  $X=\Setx{x_1,x_2,x_3}$. Then $d_{G^0}(X) \leq 6$. Since $G^0-X$ must
  be a matching on at least 3 vertices, and $G^0$ has minimum degree
  at least 3, we have $d_{G^0}(X) \geq 7$. This is a contradiction.

  (ii) Let $x,y$ be permanent vertices of $G^0$ such that $d_{G^0}(x)
  = 3$ and $d_{G^0}(y) \leq 4$. By part (i), the vertex $x$ has no
  temporary neighbour. Therefore, $x$ must have a transient neighbour
  in $G$. Since $G$ is 2-essentially 9-edge-connected,
  $G^0-\Setx{x,y}$ must be a matching and we obtain a similar
  contradiction as in the proof of (i).
\end{proof}

\begin{figure}
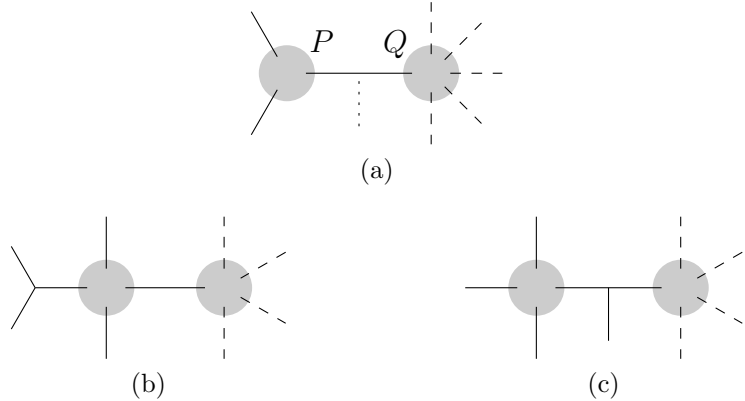
\centering
  \sfig9{}\\
  \hf\sfig{10}{}\hf\sfig{11}{}\hf
  \caption{Forbidden configurations in $H^0/\SSS$. Figure (a)
    illustrates part (i) of Lemma~\ref{l:forb}, (b) and (c) relate to
    part (ii). The gray regions represent the classes $P$ and $Q$ of
    $\SSS$, dashed lines represent optional hyperedges. The dotted
    line in figure (a) means that the hyperedge can be of size 2 or
    3.}
  \label{fig:forb}
\end{figure}

\begin{lemma}\label{l:forb}
  Let $P,Q$ be neighbouring vertices of the hypergraph $H^0/\SSS$. If
  $\size{\SSS} \geq 5$ and $G^0$ has at least 6 vertices, then the
  following holds:
  \begin{enumerate}[\quad(i)]
  \item if $d_{H^0/\SSS}(P) = 3$, then $d_{H^0/\SSS}(Q) \geq 7$,
  \item if $d_{H^0/\SSS}(P) = 4$ and $P$ is incident with a 3-hyperedge of
    $H^0/\SSS$, then $d_{H^0/\SSS}(Q) \geq 6$.
  \end{enumerate}
\end{lemma}
\begin{proof}
  We prove (i). Since $G^0$ is essentially 4-edge-connected, the class
  $P$ is trivial by Observation~\ref{obs:nontriv}; say, $P =
  \Setx{u}$. Being a permanent vertex of degree 3 in $G^0$, the vertex
  $u$ is either protected, or adjacent to a temporary vertex. 

  We set
  \begin{equation*}
    X =
    \begin{cases}
      (P\cup Q)^+ \cup \Setx{z} & \text{if $u$ is adjacent to a
        temporary vertex $z$,} \\
      (P\cup Q)^+ & \text{otherwise}.
    \end{cases}
  \end{equation*}

  Suppose that $u$ is adjacent to a temporary vertex $z$. By
  Lemma~\ref{l:path}(i), $u$ is not adjacent to any other temporary
  vertex, which implies
  \begin{equation}
    \label{eq:8}
    d_{G^0}(X) \leq 8.
  \end{equation}
  Since $G^0[X]$ is not a matching, $G^0-X$ must be a matching as
  $G^0$ is 2-essentially 9-edge-connected. 

  A similar argument shows that $G^0-X$ is a matching just as well if
  $u$ is protected. In particular, in either case, no temporary vertex
  of $G^0$ has two neighbours outside $X$.

  Enumerate the classes of $\SSS$ other than $P$ and $Q$ as
  $Y_1,\dots,Y_k$. By Observation~\ref{obs:nontriv}, each $Y_i$ is a
  trivial class, say $Y_i = \Setx{y_i}$.

  Since $k\geq 3$ and $d_{G^0}(y_i) \geq 3$ for $1\leq i \leq k$,
  inequality~\eqref{eq:8} implies that $G^0$ contains an edge joining
  two of the vertices $y_i$ --- say, $y_1$ and $y_2$. Since
  $d_{G^0}(\Setx{y_2,\dots,y_k}) \geq 3$, we have
  $d_{G^0}(\Setx{y_1,y_2})\leq 5$, so without loss of generality,
  $d_{G^0}(y_1) = 3$ and $d_{G^0}(y_2) \leq 4$. This is a
  contradiction with Lemma~\ref{l:path}(ii).

  Part (ii) can be proved using a minor modification of the above
  argument, which we leave to the reader.
\end{proof}


\section{Discharging}
\label{sec:discharging}

We continue the discussion of Section~\ref{sec:structural} by using a
discharging-type argument to prove Proposition~\ref{p:S4}. The
discharging process takes place in the hypergraph $H^0/\SSS$. Recall
our hypothesis that the partition $\SSS$ has at least 5 classes.

In Section~\ref{sec:counting}, we defined $\tau^0$ as the quasigraph
in $H^0/\SSS$ corresponding to the quasigraph $\tau = \sigma/\SSS$ in
$H^e/\SSS$. The notion of a hyperedge associated with a vertex will be
carried over from $H^e$ to $H^0/\SSS$: by definition, a hyperedge
$e/\SSS$ of $H^0/\SSS$ is associated with $P\in\SSS$ if the
corresponding hyperedge of $H^e$ is associated with a vertex of $H^e$
contained in $P$. Note that the definition makes sense thanks to the
fact that $\sigma[P]$ is a quasitree (and hence $\sigma[P]^*$ is
connected).

We write $V_i = V_i(H^0/\SSS)$. Furthermore, $V_i^\triangle$ denotes
the subset of $V_i$ consisting of vertices which have an associated
3-hyperedge and $V^\triangle$ is the union of all
$V_i^\triangle$. Given a vertex $P$ of $H^0/\SSS$, the symbol $\mnbr
P$ denotes the multiset consisting of vertices $Q$ such that $P$ and
$Q$ are contained in a hyperedge of $H^0/\SSS$, with one occurrence of
$Q$ for each such hyperedge.

We begin by assigning charges to the vertices and hyperedges of
$H^0/\SSS$, guided by inequality~\eqref{eq:main}:
\begin{itemize}
\item each vertex $P$ will get a charge of $d^0(P) - 4$ units,
\item a 3-hyperedge of $H^0/\SSS$ will get a charge of $-1$ if it is
  used by $\tau^0$ and a zero charge otherwise,
\item 2-hyperedges of $H^0/\SSS$ get zero charge.
\end{itemize}
By~\eqref{eq:main}, the total charge is negative. At the same time,
the only elements of $H^0/\SSS$ with negative charge are 3-vertices
and 3-hyperedges used by $\tau^0$.

As usual in discharging arguments, we will describe rules for the
redistribution of charge which keep the total charge unchanged and
(given the assumptions about the graph $G$) make all the individual
charges non-negative. This contradiction will show that $\sigma$ is
actually a quasitree with connected complement.

\begin{figure}
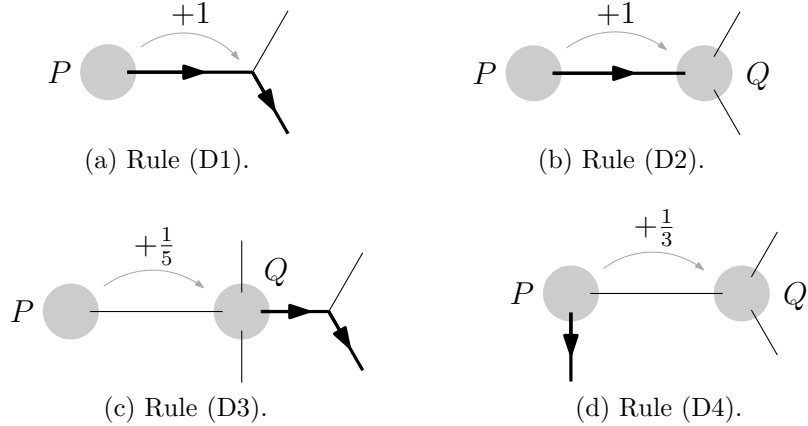

  \centering
  \hf\sfig{12}{Rule (D1).}\hf\sfig{13}{Rule (D2).}\hf\\
  \hf\sfig{15}{Rule (D3).}\hf\sfig{14}{Rule (D4).}\hf
  \caption{The rules for the redistribution of charge. Thick lines
    represent the quasigraph $\tau^0$, with the association of a
    hyperedge to a vertex shown by thick arrows. Gray arrows indicate
    the flow of the stated amount of charge. The degree of $P$ is not
    represented.}
  \label{fig:rules}
\end{figure}

Charge will only be sent by vertices, the recipient may be either a
vertex or a hyperedge. Let $P$ be a vertex of $H^0/\SSS$ (that is,
$P\in\SSS$). The rules (more than one of which may apply) are in the
following list. See Figure~\ref{fig:rules} for a schematic
representation.
\begin{enumerate}[\quad(D1)]
\item if $P\in V^\triangle$, then $P$ sends its associated hyperedge 1
  unit of charge,
\item if $P$ has an associated 2-hyperedge whose head is a degree 3
  vertex $Q$, then $P$ sends 1 unit of charge to $Q$,
\item if $P$ has a neighbour $Q$ in $V_4^\triangle$, then $P$ sends
  $Q$ a charge of $1/5$ for each common hyperedge,
\item if $P$ has a degree 3 neighbour $Q$, then $P$ sends $Q$ a charge
  of $1/3$ for each common hyperedge which is not associated with $P$.
\end{enumerate}

We claim that after the redistribution of charge, all vertices and
hyperedges of $H^0/\SSS$ will have nonnegative charge. This is clear
for 2-hyperedges and for 3-hyperedges not used by
$\tau^0$. Furthermore, each 3-hyperedge used by $\tau^0$ will obtain 1
unit of charge by rule (D1), making the resulting charge zero.

Let us therefore investigate the ways a vertex $P$ may be
discharged. Suppose first that the degree $d^0(P)$ of $P$ is at least
7. Since at most one unit of charge is transferred from $P$ based on
rules (D1) and (D2), the total transfer from $P$ is at most
\begin{equation*}
  1 + (d^0(P) - 1)\cdot\frac13 \leq d^0(P)-4,
\end{equation*}
where the inequality follows from the assumption that $d^0(P) \geq
7$. Since the right hand side is the original charge of the vertex,
the resulting charge is nonnegative.

We may thus assume that $d^0(P) \leq 6$. By Lemma~\ref{l:forb}(i), no
neighbour $P'$ of $P$ in $H^0/\SSS$ has $d^0(P') = 3$. Suppose that
$d^0(P) = 6$. If $P$ is discharged according to rule (D1) or (D2),
then rule (D4) does not apply to $P$, and the transfer from $P$ is at
most $1 + 5 \cdot 1/5 = 2$, the initial charge of $P$. On the other
hand, if none of (D1) and (D2) apply, then $P$ sends at most $6 \cdot
1/3 = 2$ units of charge as well.

If $d^0(v) = 5$, then $P$ has no neighbour in $V_4^\triangle$
(Lemma~\ref{l:forb}(ii)), which rules out the use of (D3) for the
discharging of $P$. Furthermore, the applicability of (D1), (D2) and
(D4) is mutually exclusive. This means that $P$ only sends at most a
charge of 1 unit, which equals its initial charge.

We are left with the case that $d^0(P) \leq 4$. Suppose that $d^0(P) =
4$ (so its initial charge is zero). Lemma~\ref{l:forb} implies that no
neighbour of $P$ in $H^0/\SSS$ is contained in $V_3\cup
V_4^\triangle$. Thus, if $P$ sends any charge at all, it must be
according to rule (D1). In this case, $P\in V_4^\triangle$ and
according to rule (D3), $P$ receives a charge of $1/5$ from each of
the five vertices in $\mnbr P$, so its resulting charge is $-1 + 5
\cdot 1/5 = 0$.

It remains to consider the case that $d^0(P) = 3$. By
Lemma~\ref{l:g0}(ii), $P$ contains a single vertex $v$ of $H^0$. By
the property (Q2) of $\sigma$, $v$ is not a leaf of $\sigma^*$ in
$H^e$, and therefore $P$ is not a leaf of $(\tau^0)^*$ in
$H^0/\SSS$. Let $e$ be an edge of $(\tau^0)^*$ with $P$ as its head, and
let the tail of $e$ be denoted by $t$.

Similarly to the discussion in the preceding cases, $P$ is not
discharged according to any of the rules (D2), (D3) or (D4). We
distinguish two cases.

If $P\notin V_3^\triangle$, then it does not send any charge to its
neighbours or incident hyperedges. By Lemma~\ref{l:forb}, at most one
of the four vertices in $\mnbr P$ has an associated 3-hyperedge $f$
such that the head of $\tau^0(f)$ is $P$. Each of the remaining vertices
sends either $1/3$ or $1$ unit of charge to $P$ (according to rule
(D4) or (D2), respectively), which accounts for a resulting charge of
at least $-1 + 3 \cdot 1/3 = 0$.

If $P\in V_3^\triangle$, then $P$ is discharged according to (D1),
which decreases its charge from $-1$ to $-2$. On the other hand,
Lemma~\ref{l:forb} implies that $e$ is a 2-hyperedge, so $t$ sends one
unit of charge to $P$ according to rule (D2). Furthermore, as above,
$P$ gets at least a charge of $1/3$ from each of the three remaining
vertices in $\mnbr P$. Hence, the new charge is nonnegative
again. This concludes the analysis. The contradiction establishes
Proposition~\ref{p:S4}.


\section{Completing the proof of Theorem~\ref{t:main}}
\label{sec:final}

In this section, we prove Theorem~\ref{t:main}. Before doing so, we
narrow down the set of possible cases by proving that the size of the
partition $\SSS$ is not greater than 2. 

Recall from Section~\ref{sec:counting} the notation $\widetilde m^0_i$
for $\size{E_i(H^0/\SSS)}$ (where $i\in\Setx{2,3}$) and
inequality~\eqref{eq:small}:
\begin{equation*}
  \widetilde m^0_2 + \widetilde m^0_3 \leq 2n-3 + \eps/2,
\end{equation*}
where $n=\size{\SSS}$ and $\eps$ is the difference of the sum of
vertex degrees in $H^0/\SSS$ and in $H^e/\SSS$. Since $\eps\leq 4$
and, by Proposition~\ref{p:S4}, $n \leq 4$, we have
\begin{equation}
  \label{eq:few}
  H^0/\SSS \text{ has at most 7 hyperedges}. 
\end{equation}

Before stating the next proposition, we recall that nontrivial class
of $\SSS$ is a class $X$ such that $\size{X^+} \geq 1$.
\begin{proposition}\label{p:S2}
  If $G^0$ has at least 6 vertices, then $n\leq 2$. Moreover, if $n =
  2$, then the following hold:
  \begin{enumerate}[\quad(i)]
  \item one of the classes of $\SSS$ is a trivial class $\Setx x$,
  \item the degree of $x$ in $H^0$ is 3,
  \item $x$ is incident with the hyperedges $k(e_1)$ and $k(e_2)$ of
    $H^0$,
  \item the size of $k(e_1)$ and $k(e_2)$ is 2.
  \end{enumerate}
\end{proposition}
\begin{proof}
  The proof consists of a series of claims.
  \setcounter{claim}{0}
  \begin{claim}
    The partition $\SSS$ has at most one nontrivial class.
  \end{claim}
  \begin{claimproof}
    By Observation~\ref{obs:nontriv}, if $\SSS$ has two nontrivial
    classes $X,Y$, \eqref{eq:few} implies that $\bd{G^0}{X^+}$ is a
    2-essential edge-cut in $G^0$ of size at most 7.
  \end{claimproof}

  \begin{claim}
    If $\SSS$ has a nontrivial class, then $H^0/\SSS$ contains no
    3-hyperedge.
  \end{claim}
  \begin{claimproof}
    Let $X\in\SSS$ be nontrivial and let $e/\SSS$ be a 3-hyperedge of
    $H^0/\SSS$ (where $e$ is a hyperedge of $H^0$). We can choose two
    vertices of $e/\SSS$ (say $Y_1,Y_2$) distinct from $X$. For
    $i=1,2$, let $y_i$ be the vertex of $e$ in $Y_i$, and let $w$ be
    the vertex of $G^0$ such that $e = h(w)$. Then the edge-cut
    $\bd{G_0}{X^+}$ of size at most 7 separates $X^+$ from the path
    $y_1wy_2$ in $G^0$ and is therefore 2-essential, a contradiction.
  \end{claimproof}

  \begin{claim}
    The partition $\SSS$ has exactly one nontrivial class.
  \end{claim}
  \begin{claimproof}
    For contradiction, suppose that all the classes of $\SSS$ are
    trivial. Recall the parameter $\widetilde m^0_3$, introduced above
    equation~\eqref{eq:main}. Since we assume that $G^0$ has at least
    6 vertices, we must have $\widetilde m^0_3 \geq 2$. We let the
    vertices of $H^0$ be denoted by $x,y,z$ or $x,y,z,u$ depending on
    whether $n$ equals 3 or 4. Furthermore, the 3-hyperedges of $H^0$
    are denoted by $h(w_i)$, where $1\leq i\leq \widetilde m^0_3$ and
    all the $w_i$'s are temporary vertices of $G^0$.
    
    If $H^0$ contains two 3-hyperedges intersecting in exactly two
    vertices (say, $e=xyz$ and $f=xyu$), then the edge-cut
    $\bd{G^0}{\Setx{x,z}^+}$ is a 2-essential edge-cut of size at most
    7, a contradiction. Thus, all the 3-hyperedges of $H^0$ contain the
    same triple of vertices, say $\Setx{x,y,z}$.

    Suppose now that $\widetilde m^0_3 \geq 3$. By~\eqref{eq:few}, we may assume
    that $x$ is incident with only at most three 2-hyperedges of
    $H^0$. But then the size of the 2-essential edge-cut
    $\bd{G^0}{\Setx{x,w_1,w_2}^+}$ in $G^0$ is at most $8$, a
    contradiction.
    
    We conclude that $\widetilde m^0_3 = 2$, which implies $n=4$. If $uz$ is a
    2-hyperedge in $H^0$, then the edge-cut $\bd{G^0}{\Setx{x,w_1,y}^+}$
    in $G^0$ separates the paths $xw_1y$ and $uzw_2$ and is therefore
    2-essential. In addition, its size is at most 8. This contradiction
    concludes the proof of the claim.
  \end{claimproof}

  \begin{claim}\label{cl:two-classes}
    The partition $\SSS$ has at most two classes.
  \end{claim}
  \begin{claimproof}
    Suppose that $n > 2$. By the above claims, $\SSS$ has a single
    nontrivial class $P$ and all the hyperedges of $H^0$ are of size
    2. Let the vertices of $H^0$ comprising the nontrivial parts of
    $\SSS$ be denoted by $x,y$ or $x,y,z$ depending on $n$.

    Suppose first that $n=4$. Since the degree of each of the vertices
    $x,y,z$ in $H^0$ is at least 3, \eqref{eq:few} implies that at least
    two 2-hyperedges have both endvertices in the set
    $\Setx{x,y,z}$. Consequently, the edge-cut $\bd{G^0}{P^+}$ of size
    at most 7 is 2-essential.

    We infer that $n=3$. By a similar argument, the vertices $x$ and $y$
    must be adjacent vertices of degree 3 in $G^0$. Since they are
    permanent and there is no temporary vertex, they must be
    protected. However, if $x$ is adjacent to a transient vertex, then
    it is easy to show that $\bd{G^0}{P^+}$ is a 2-essential 4-edge-cut
    in $G^0$.
  \end{claimproof}
  
  To finish the proof of the proposition, it remains to establish
  properties (ii)--(iv). Let us write $\SSS = \Setx{P,\Setx x}$, where
  $P$ is the nontrivial class. By inequality~\eqref{eq:small} and the
  fact that $G^0$ is 3-edge-connected, $x$ has degree 3 in $H^0$ and
  $\eps=4$. The latter fact means that $k(e_1)$ and $k(e_2)$ are
  2-hyperedges of $H^0$ incident with $x$. The proof is complete.
\end{proof}

Having established Proposition~\ref{p:S2}, we can now prove the main
result of this paper.
\begin{proof}[Proof of Theorem~\ref{t:main}]
  If the graph $G^0$ has at most 5 vertices, then $L(G)$ is
  Hamilton-connected by Lemma~\ref{l:small}. Assume thus that
  $\size{V(G^0)} \geq 6$. By Proposition~\ref{p:S2}, we have $n \leq
  2$.

  If $n=1$, then by the choice of $\SSS$, $\sigma$ is an acyclic
  quasigraph in $\He$ that is both connected and anticonnected on
  $V(\He)$. Proposition~\ref{p:qt-join} implies that $G(\He)$ admits
  an $a_1a_2$-trail spanning $V(\He)$. By Lemma~\ref{l:join}, it
  follows that $G$ admits an internally dominating
  $(e_1,e_2)$-trail. Since the choice of $e_1$ and $e_2$ is arbitrary,
  $L(G)$ is Hamilton-connected by Theorem~\ref{t:ham-conn-preimage}.

  The discussion in the case that $n=2$ is only slightly more
  complicated. By Proposition~\ref{p:S2}, $\SSS$ contains a trivial
  class $\Setx{x}$, $x$ has degree 3 in $H^0$ and is incident in $H^0$
  with the 2-hyperedges $k(e_1)$ and $k(e_2)$. Let $P$ denote the
  other class of $\SSS$ and let $y_1$ and $y_2$ be the endvertex of
  $k(e_1)$ and $k(e_2)$, respectively, in $P$. Furthermore, let $f$ be
  the third 2-hyperedge of $H^0$ incident with $x$, and let $y_3$ be
  its endvertex in $P$. 

  By the definition of the temporary vertices $a_1$ and $a_2$, we may
  assume without loss of generality that $a_1 = x$ and $a_2 = y_2$. To
  find an $xy_2$-trail in $G(\He)$ spanning $V(\He)$,
  we proceed as follows.

  Denoting the hypergraph obtained from $\He$ by removing $x$ by
  $H_1$, we observe that $\sigma$ determines an acyclic quasigraph in
  $H_1$ that is both connected and anticonnected on $V(H_1)=P$. By
  Proposition~\ref{p:qt-join}, $G(H_1)$ admits an $y_3y_2$-trail $T_1$
  spanning $P$. Adding the edge $f$ to the beginning of $T_1$, we
  obtain an $xy_2$-trail in $G(\He)$ spanning $V(\He)$ as desired.
\end{proof}


\section{Claw-free graphs}
\label{sec:claw}

As with Theorem~\ref{t:lai}, Theorem~\ref{t:main} can be extended to
claw-free graphs. The procedure is the same as that used in
\cite[Section 11]{KV-hamilton}. We use the \emph{$M$-closure}
introduced in~\cite{RV-line}, namely the following
result~\cite[Theorem 9]{RV-line}:

\begin{theorem}\label{t:closure}
  If $G$ is a connected claw-free graph, then there is a well-defined
  graph $cl^M(G)$ with the following properties:
  \begin{enumerate}[\quad(i)]
  \item $G$ is a spanning subgraph of $cl^M(G)$,
  \item $cl^M(G)$ is the line graph of a multigraph,
  \item $cl^M(G)$ is Hamilton-connected if and only if $G$ is
    Hamilton-connected.
  \end{enumerate}
\end{theorem}

In the context of the present paper, the reference to multigraphs in
Theorem~\ref{t:closure} is not necessary, since parallel edges in
graphs are allowed by default.

By condition (i) in Theorem~\ref{t:closure}, the connectivity of
$cl^M(G)$ is greater than or equal to that of $G$. As shown by the
following observation, the closure operation does not decrease the
essential connectivity either.

\begin{lemma}\label{l:ess}
  If $G$ is essentially $k$-connected, then $cl^M(G)$ is also
  essentially $k$-connected.
\end{lemma}
\begin{proof}
  Clearly, $\size{V(cl^M(G))} > k$. For contradiction, let $X$ be a
  minimal essential vertex-cut in $cl^M(G)$ of size less than
  $k$. Since $X$ is not essential in $G$, there is a component $K$ of
  $cl^M(G)-X$ such that $K$ contains some edges, but $G[V(K)]$ is
  edgeless.

  The operation $cl^M$, as defined in~\cite[Section~4]{RV-line},
  consists of a sequence of local completions at suitable vertices
  $x_1,\dots,x_\ell$. Here, the \emph{local completion} at $x_i$ is
  the addition of all possible edges joining the neighbours of
  $x_i$. Let $e$ be an edge of $K$ and let $Y = X \cap
  \Setx{x_1,\dots,x_\ell}$. In $G$, all the vertices with a neighbour
  in $V(K)$ are contained in $X$. The fact that $K$ contains at least
  one edge implies that $Y\neq\emptyset$. Let us say that $x_i\in Y$,
  where $1\leq i\leq\ell$. By the minimality of $X$, $x_i$ has a
  neighbour (in $cl^M(G)$) in some component $L$ of $cl^M(G)-X$ other
  than $K$. Although the edge between them could be added by a local
  completion at some vertex $x_j$ ($1\leq j\leq\ell$), this can only
  happen if some vertex of $Y$ has a neighbour in $V(L)$ prior to the
  local completion. We conclude that some vertex of $Y$, say $x_k$,
  has a neighbour in $V(L)$ in $G$. But then the local completion at
  $x_k$ adds an edge between a neighbour of $x_k$ in $V(K)$ and a
  neighbour of $x_k$ in $V(L)$, contradicting the assumption that $K$
  and $L$ are different components of $cl^M(G)-X$.
\end{proof}

Using Lemma~\ref{l:ess}, we find that $cl^M(G)$ is a 3-connected,
essentially 9-connected line graph. Thus, $cl^M(G)$ is
Hamilton-connected by Theorem~\ref{t:main}. Condition (iii) of
Theorem~\ref{t:closure} implies that $G$ is Hamilton-connected.


\section{Conclusion}
\label{sec:conclusion}

We have shown that every 3-connected, essentially 9-connected
claw-free graph is Hamilton-connected, and that this assertion is
false with 9 replaced by 4. The obvious question is left unresolved:
what is the least value of $k$ such that 3-connected, essentially
$k$-connected claw-free graphs are Hamilton-connected (or
hamiltonian)? This remains an interesting problem for further
investigation.



\begin{thebibliography}{99}

\bibitem{Die-graph}%
  R. Diestel, Graph Theory, 3rd ed., Springer, 2005.

\bibitem{KV-hamilton}%
  T. Kaiser and P. Vr\'{a}na, Hamilton cycles in 5-connected line
  graphs, European J. Combin. 33 (2012), 924--947.

\bibitem{KV-quasi}%
  T. Kaiser and P. Vr\'{a}na, Quasigraphs and skeletal partitions,
  submitted for publication.

\bibitem{LSWZ-every}%
  H.-J. Lai, Y. Shao, H. Wu and J. Zhou, Every 3-connected,
  essentially 11-connected line graph is Hamiltonian,
  J. Combin. Theory Ser. B 96 (2006), 571--576.

\bibitem{LLS-s-hamiltonian}%
  H.-J. Lai, Y. Liang and Y. Shao, On $s$-hamiltonian-connected line
  graphs, Discrete Math. 308 (2008), 4293--4297.

\bibitem{LLZ-eulerian}%
  D. Li, H.-J. Lai and M. Zhan, Eulerian subgraphs and
  Hamilton-connected line graphs, Discrete Appl. Math. 145 (2005),
  422--428.

\bibitem{LY-every}%
  H. Li and W. Yang, Every 3-connected essentially 10-connected line
  graph is Hamilton-connected, Discrete Math. 312 (2012), 3670--3674.

\bibitem{NW-edge}%
  C. St. J. A. Nash-Williams, Edge-disjoint spanning trees of finite
  graphs, J.~London Math. Soc. (1961), 445--450.

\bibitem{RV-line}%
  Z. Ryj\'{a}\v{c}ek and P. Vr\'{a}na, Line graphs of multigraphs and
  Hamilton-connectedness of claw-free graphs, J. Graph Theory 66
  (2011), 152--173.

\bibitem{Sha-claw}%
  Y. Shao, Claw-free graphs and line graphs, Ph.D. Dissertation, West
  Virginia University, 2005.

\bibitem{Tut-problem}%
  W. T. Tutte, On the problem of decomposing a graph into $n$
  connected factors, J. London Math. Soc. 36 (1961), 221--230.

\bibitem{Tho-hypohamiltonian}%
  C. Thomassen, Hypohamiltonian and hypotraceable graphs, Discrete
  Math. 9 (1974), 91--96.

\bibitem{Tho-reflections}%
  C. Thomassen, Reflections on graph theory, J. Graph Theory 10
  (1986), 309--324.

\bibitem{YLLG-collapsible}%
  W. Yang, H.-J. Lai, H. Li and X. Guo, Collapsible graphs and
  Hamiltonian connectedness of line graphs, Discrete Applied Math. 160
  (2012), 1837--1844.

\bibitem{YXLG-hamiltonicity}%
  W. Yang, L. Xiong, H. Lai, X. Guo, Hamiltonicity of 3-connected line
  graphs, Appl. Math. Lett. 25 (2012), 1835--1838.

\end{thebibliography}
\end{document}